\documentclass[11pt]{article}
\usepackage{amsmath}
\usepackage{amsthm}
\usepackage{amssymb}
\usepackage{amscd}
\usepackage{relsize}
\usepackage{enumerate}
\newtheorem{dfn}{Definition}
\newtheorem{teo}{Theorem}

\newtheorem{cor}[teo]{Corollary}

\newcommand{\sinc}{\operatorname{sinc}}
\DeclareMathOperator*{\esup}{ess\,sup}
\DeclareMathOperator*{\einf}{ess\,inf}

\title{{\bf Sampling in the range of the analysis operator of a continuous frame having unitary structure}\thanks{The author is very pleased to dedicate this work to his friend, the mathematician Lance L. Littlejohn on the occasion of his 70th birthday. }}
\author{
{\bf Antonio G. Garc\'{\i}a}\thanks{E-mail:\texttt{agarcia@math.uc3m.es}}
}
\date{}
\begin{document}
\maketitle
\begin{itemize}
\item[] Departamento de Matem\'aticas, Universidad Carlos III de Madrid, Spain.
\end{itemize}
\begin{abstract}
We establish a regular sampling theory in the range of the analysis operator of a continuous frame having a unitary structure. The unitary structure is related with a unitary representation of a locally compact abelian group on a separable Hilbert space. The samples are defined by means of suitable discrete convolution systems which generalize some usual sampling settings; here regular sampling means that the  samples are taken at a countable discrete subgroup.
\end{abstract}

\noindent{\bf AMS}: 42C15; 94A20; 22B05; 20H15.
\section{Statement of the problem}
\label{section1}
In this paper a regular sampling theory in the range space of the analysis operator of a continuous frame is established. To be more specific, the continuous frame has a unitary structure associated with a unitary representation $t\in G \mapsto U(t)$ of an LCA (locally compact abelian) group $G$ on a separable Hilbert space $\mathcal{H}$.  The functions to be recovered, at a countable discrete subgroup $H$ of $G$, are associated with functions in a unitary $H$-invariant subspace $\mathcal{H}_\Phi$ of $\mathcal{H}$ (that, eventually, could coincide with $\mathcal{H}$); the subscript $\Phi=\{ \varphi_1, \varphi_2, \dots, \varphi_N \} \subset \mathcal{H}$ denotes a set of stable generators for $\mathcal{H}_\Phi$. Thus, the functions to be recovered from a sequence of samples at $H$ look like: 
\[
F_f(t)=\big\langle f, U(t)\phi \big\rangle_\mathcal{H}\,, \,\, t\in G\,, \, \text{ where $f\in\mathcal{H}_\Phi $}\,.
\]
The set of these functions $F$ (we omit the subscript) forms a reproducing kernel Hilbert space (in short RKHS) of  continuous and bounded functions contained in 
$L^2(G)$ that will be denoted as $\mathcal{H}_{U,\phi,\Phi}$. For some examples of these spaces we cite, among others,  the Paley-Wiener spaces $PW_{\pi\sigma}$, shift-invariant subspaces $V_\Phi^2$ in $L^2(\mathbb{R}^d)$ and, in the non-abelian case, the range space of the continuous wavelet or Gabor transforms.

\medskip

The goal in this work is to recover, in a stable way, any $F$ in $\mathcal{H}_{U,\phi,\Phi}$ from a sequence of samples taken at the subgroup 
$H$. For instance, for the sequence of {\em pointwise samples} $\{F(t)\}_{t\in H}$  we obtain a suitable expression as the output of a discrete convolution system (see Eq.~\eqref{s1} in Section \ref{section3})
\[
F(t)=\sum_{n=1}^N \big(a_n \ast_{_H} x_n\big)(t)\,, \,\, t\in H\,,
\]
where $a_n(t):=\big\langle \varphi_n, U(t)\phi\big\rangle_\mathcal{H}$, $t\in H$,  belongs to 
$\ell^2(H)$ for $n=1, 2,\dots, N$, $x_n\in \ell^2(H)$, $n=1, 2,\dots, N$, are the coefficients of the expansion $f=\sum_{n=1}^N\sum_{t\in H} x_{n}(t) U(t)\varphi_{n}$, and $\ast_{_H}$ denotes the convolution in $\ell^2(H)$. The same occurs for the {\em average samples} $\{\mathcal{M}_m F(t)\}_{t\in H}$ defined by
\[
\mathcal{M}_mF(t):=\big\langle f, U(t)\psi_m\big\rangle_\mathcal{H},\,\, t\in H\,, \text{ for some fixed $\psi_m\in \mathcal{H}$}\,.
\]
The above examples lead us to define, for any $F$ in the space 
$\mathcal{H}_{U,\phi,\Phi}$, a sequence of {\em generalized samples} $\{\mathcal{L}_m F(t)\}_{t\in H;\,m=1,2,\dots,M}$, at the subgroup $H$, by means of an $M\times N$ matrix $A=[a_{m,n}]$ with entries in  $\ell^2(H)$ as
\[
\mathcal{L}_mF(t):=\sum_{n=1}^N \big(a_{m,n}\ast_{_H} x_n\big)(t)\,, \quad t\in H\,, \quad m=1,2, \dots,M\,.
\]
Thus, under appropriate conditions (see Definition \ref{sampdef} in Section \ref{section3}), the main sampling result (see Theorem \ref{sampteo} Section \ref{section4}) proves that  there exist $M$ sampling functions $S_m$ in $\mathcal{H}_{U,\phi,\Phi}$, $m=1, 2,\dots, M$, such that the sequence $\{S_m(\cdot-t)\}_{t\in H;\,m=1, 2,\dots,M}$ is a frame for 
$\mathcal{H}_{U,\phi,\Phi}$, and the sampling expansion
\[
F(s)=\sum_{m=1}^M\sum_{t\in H} \mathcal{L}_m F(t)\,S_m(s-t)\,, \quad s\in G\,,
\]
holds for every $F\in \mathcal{H}_{U,\phi,\Phi}$. In addition, the sampling functions $S_m$, $m=1, 2,\dots,M$, can be obtained, via the matrix $A$, by means of an explicit method (see in the end of Section \ref{section3}). 

\medskip

The used mathematical techniques are similar to those in Ref.~\cite{garcia:19}. They lie in exploiting the relationship between discrete convolution systems and frames of translates in the product Hilbert space $\ell^2_{_N}(H):=\ell^2(H)\times  \dots \times \ell^2(H)$ ($N$ times); this is an auxiliary space isomorphic to  $\mathcal{H}_{U,\phi,\Phi}$.

\medskip

The work is organized as follows: In Section \ref{section2} we include some needed preliminaries on continuous and discrete frames; on Fourier analysis for a countable discrete group, and on convolution systems in the Hilbert space 
$\ell^2_{_N}(H)$.  It is worth to mention the relationship between convolution systems in $\ell^2_{_N}(H)$ and frames of translates in $\ell^2_{_N}(H)$ showing the equi\-va\-lence of their properties. The needed results have been borrowed from Refs.~\cite{garcia:19,gerardo:19}.
Section \ref{section3} is devoted to introduce the subspace of $L^2(G)$ where the sampling theory will be carried out. Finally,  Section \ref{section4} contains the main sampling result along with some pertinent comments and remarks. Although the work deals with abelian groups, an example involving semi-direct products of groups is included; this particular non-abelian case will be treated by using the theory developed in Section \ref{section4}. 

It should be noted that working in LCA groups is not just a unified way of dealing with the classical groups 
$\mathbb{R}^d, \mathbb{Z}^d, \mathbb{T}^d, \mathbb{Z}_s^d$: signal processing often involves products of these groups which are also locally compact abelian groups. For example, multichannel video signal involves the group $\mathbb{Z}^d \times \mathbb{Z}_s$, where $d$ is the number of channels and $s$ the number of pixels of each image. Finally, some companion references in sampling theory are, for instance, Refs.~\cite{aldroubi:05,bhandari:12,hector:14,kang:11,pohl:12,shang:07}.
\section{Some preliminaries}
\label{section2}
\subsection{Continuous and discrete frames}
\label{section2-1}
Let $\mathcal{H}$ be a Hilbert space and let $(\Omega , \mu)$ be a measure space. A mapping $\psi:\Omega \rightarrow \mathcal{H}$ is a {\em continuous frame} for 
$\mathcal{H}$ with respect to $(\Omega , \mu)$ if $\psi$ is weakly measurable, i.e., for each $x\in \mathcal{H}$ the function $w\mapsto \langle x, \psi(w)\rangle$ is measurable, and there exist constants $0<A\le B$ such that
\begin{equation}
\label{defcf}
A\|x\|^2 \le \int_\Omega \big|\langle x, \psi(w)\rangle \big|^2 d\mu(w) \le B\|x\|^2 \quad \text{for each $x\in \mathcal{H}$}\,.
\end{equation}
The constants $A$ and $B$ are the lower and upper continuous frame bounds respectively. The mapping $\psi$ is a {\em tight continuous frame} if $A=B$; a {\em Parseval continuous frame} if $A=B=1$. The mapping $\psi$ is called a {\em Bessel family} if only the right-hand inequality holds. Throughout this paper we refer a continuous frame as the mapping $\psi:\Omega \rightarrow \mathcal{H}$, or as the family $\{\psi(w)\}_{w\in \Omega}$, or 
$\{\psi_w\}_{w\in \Omega}$, in the Hilbert space $\mathcal{H}$.

\medskip

There are a lot of examples of continuous frames in the mathematical/physics literature. For instance: the family $\{k_w\}_{w\in \Omega}$ of  reproducing kernels of a RKHS $\mathcal{H}_k$ contained in $L^2(\Omega,\mu)$ is a continuous Parseval frame with respect to $(\Omega, \mu)$;  a {\em Gabor system} $\big\{ M_\xi T_x g \,:\, (x, \xi)\in \mathbb{R}^d\times \mathbb{R}^d\big\}$ is a tight continuous frame for $L^2(\mathbb{R}^d)$ with respect to $\big(\mathbb{R}^d\times \mathbb{R}^d, dx\,d\xi \big)$, where $g\in L^2(\mathbb{R}^d)$ is a fixed non zero function, and $M_\xi$ and $T_x$ denote the modulation and translation operators in $L^2(\mathbb{R}^d)$ respectively;  a {\em wavelet system} $\big\{T_bD_a\psi \,:\, (a, b)\in (\mathbb{R}\setminus \{0\})\times \mathbb{R}\big\}$ is a tight continuous frame for $L^2(\mathbb{R})$ with respect to $\big((\mathbb{R}\setminus \{0\})\times \mathbb{R}, \dfrac{dadb}{a^2}\big)$, where $T_b$ and $D_a$ denote the translation and dilation operators in $L^2(\mathbb{R})$ respectively, and $\psi \in L^2(\mathbb{R})$ is an admissible function, i.e., a function for which the constant $c_{_\psi}:=\int_\mathbb{R} \frac{|\widehat{\psi}(w)|^2}{|w|} dw<+\infty$; {\em coherent states} in physics, etc. (see, for instance, Refs.~\cite{antoine:93,ole:16,fornasier:05}).

\medskip

The operator $T_\psi: L^2(\Omega,\mu) \rightarrow \mathcal{H}$ weakly defined for each $f\in L^2(\Omega,\mu)$ by 
\[
\langle T_\psi f, x\rangle=\int_\Omega f(w) \big\langle \psi(w), x\big\rangle d\mu(w)\,, \quad x\in \mathcal{H}\,,
\] 
is linear and bounded; it is called the {\em synthesis operator} of $\{\psi_w\}_{w\in \Omega}$. Its adjoint operator $T_\psi^*:\mathcal{H}  \rightarrow L^2(\Omega,\mu)$ is given by $(T_\psi^*\,x)(w)=\langle x, \psi(w)\rangle$, $w\in \Omega$, and it is called the {\em analysis operator} of $\{\psi_w\}_{w\in \Omega}$. The {\em continuous frame operator} $S_\psi=T_\psi \, T_\psi^*$ is a bounded, self-adjoint, positive and invertible operator in $\mathcal{H}$. For any $x\in \mathcal{H}$ we have the weak representations
\[
x=\int_\Omega  \big\langle x, \psi(w) \big\rangle S_\psi^{-1}\psi(w) d\mu(w)=\int_\Omega  \big\langle x, S_\psi^{-1}\psi(w) \big\rangle \psi(w) d\mu(w)\,.
\]
The counting measure $\mu$ on $\Omega=\mathbb{N}$ gives the classical definition of  discrete frame $\{x_n\}_{n=1}^\infty$: there exist two constants $0<A\le B$ such that
\begin{equation}
\label{defdf}
A\|x\|^2 \le \sum_{n=1}^\infty |\langle x, x_n \rangle |^2 \le B\|x\|^2 \quad \text{for each $x\in \mathcal{H}$}\,.
\end{equation}
Given a discrete frame $\{x_n\}_{n=1}^\infty$ for $\mathcal{H}$, its {\em preframe (synthesis) operator} $T: \ell^2(\mathbb{N}) \rightarrow \mathcal{H}$  is defined by $T\big(\{c_n\}_{n=1}^\infty\big)=\sum_{n=1}^\infty c_n \, x_n$. Its adjoint operator $T^* : \mathcal{H} \rightarrow \ell^2(\mathbb{N})$ is given by $T^*\,x= \{\langle x, x_n\rangle \}_{n=1}^\infty$, and it is called its {\em analysis operator}. The {\em frame operator} $S$ is defined by $S(x):=T \,T^*\,x= \sum_{n=1}^\infty \langle x, x_n\rangle \, x_n$, \,$x\in \mathcal{H}$; it is a bounded, invertible, positive and self-adjoint operator in $\mathcal{H}$. The sequence $\{S^{-1}x_n\}_{n=1}^\infty$ is also a frame for $\mathcal{H}$ called the {\em canonical dual frame}. For each $x\in \mathcal{H}$ we have the expansions
\[
x=\sum_{n=1}^\infty \langle x, x_n\rangle \, S^{-1}x_n=\sum_{n=1}^\infty \langle x, S^{-1}x_n\rangle \, x_n\,.
\]
As a consequence, given a frame $\{x_n\}_{n=1}^\infty$ for $\mathcal{H}$ the representation property of any vector $x\in \mathcal{H}$ as a series $x=\sum_{n=1}^\infty c_n \,x_n$ is retained, but, unlike the case of Riesz (orthonormal) bases, the uniqueness of this representation is sacrificed. 
Suitable frame coefficients $\{c_n\}$ which depend continuously and linearly on $x$ are obtained by using the {\em dual frames} $\{y_n\}_{n=1}^\infty$ of $\{x_n\}_{n=1}^\infty$, i.e., $\{y_n\}_{n=1}^\infty$ is another frame for $\mathcal{H}$ such that 
\begin{equation*}
\label{expansion}
x=\sum_{n=1}^\infty \langle x, y_n \rangle \,x_n=\sum_{n=1}^\infty \langle x, x_n \rangle \,y_n\quad \text{ for each $x\in \mathcal{H}$}\,.
\end{equation*} 
In particular, frames include orthonormal and Riesz bases for $\mathcal{H}$. For more details and proofs on discrete and continuous frames, see, for instance, Refs. \cite{antoine:93,ole:16,fornasier:05,gabardo:03,rahimi:06}. 

\medskip

Assume that $\{\psi(w)\}_{w\in \Omega}$ is a continuous frame for a Hilbert space $\mathcal{H}$ with respect to $(\Omega, \mu)$ such that the mapping $w\mapsto \psi(w)$ is weakly continuous, i.e., for each $x\in \mathcal{H}$ the function $w\mapsto \langle x, \psi(w)\rangle$ is continuous. Its analysis operator $T_\psi^*: \mathcal{H} \rightarrow L^2(\Omega, \mu)$ is a bounded and boundedly invertible operator on its range denoted as $\mathcal{H}_\psi :=\text{Range }T_\psi^*$. This is a closed subspace of $L^2(\Omega, \mu)$ described as the functions $F_x$ such that
\[
\begin{array}[c]{ccll}
 & \mathcal{H} & \longrightarrow & \mathcal{H}_\psi\\
       & x & \longmapsto & F_x \,\,:\,\, F_x(w)=\big\langle x, \psi(w)\big\rangle_{\mathcal{H}}\,, \quad w\in \Omega\,.
\end{array}
\]
Besides $\mathcal{H}_\psi$ is a RKHS (of continuous functions in $\Omega$) whose reproducing kernel is given by 
\[
k_\psi (u,v)=\big\langle \psi(v), S_\psi^{-1}\psi(u)\big\rangle_{\mathcal{H}}\,, \quad u,v\in \Omega\,,
\]
where $S^{-1}_\psi$ denotes the inverse of the frame operator $S_\psi$ associated to $\{\psi(w)\}_{w\in \Omega}$. That is, for each $F_x\in \mathcal{H}_\psi$ we have the reproducing property
\[
F_x(u)=\int_\Omega F_x(v)\,k_\psi (u,v)\,d\mu(v)=\big\langle F_x, k_\psi (\cdot,u) \big\rangle_{L^2(\Omega, \mu)}\,, \quad u\in \Omega\,.
\]
\subsection{Discrete convolution systems and frames of translates}
\label{section2-2}
Let $(H, +)$ be a countable discrete abelian group and let $\mathbb{T}=\{z\in \mathbb{C}: |z|=1\}$ be the unidimensional torus. A {\em character $\xi$} of $H$ is a homomorphism $\xi:H \mapsto \mathbb{T}$, i.e., $\xi(h+h')=\xi(h)\xi(h')$  for all $h,h'\in H$; we denote $\xi(h)=(h,\xi)$.  By defining $(\xi+\xi')(h)=\xi(h)\xi'(h)$,  the set of characters $\widehat{H}$ is a group, called the {\em dual group} of $H$; since $H$ is discrete, the group $\widehat{H}$ is compact \cite[Prop. 4.4]{folland:95}. 

For $x\in \ell^1(H)$  its {\em Fourier transform} is defined by
\[
\widehat{x}(\xi):=\sum_{h\in H}x(h) \overline{(h,\xi )}=\sum_{h\in H}x(h) (-h,\xi )\,,\quad \xi\in \widehat{H}\,.
\]
The Plancherel theorem extends uniquely the Fourier transform on $\ell^1(H)\cap \ell^2(H)$ to a unitary isomorphism from $\ell^2(H)$ to $L^2(\widehat{H})$. For the details see, for instance, Refs.~\cite{folland:95,fuhr:05}.

\medskip

Let $H$ be a countable discrete group and let consider the product Hilbert space $\ell^2_{_N}(H):=\ell^2(H)\times \dots \times \ell^2(H)$ ($N$ times). For a matrix $A=[a_{m,n}] \in \mathcal{M}_{_{M\times N}}\big(\ell^2(H)\big)$, i.e., an $M\times N$ matrix with entries in $\ell^2(H)$, and $\mathbf{x} \in \ell^2_{_N}(H)$, the {\em matrix convolution} $A \ast \mathbf{x}$ in $H$ is defined by
\[
(A \ast \mathbf{x})(h):=\sum_{h'\in H} A(h-h')\, \mathbf{x}(h'),\quad h\in H\,.
\]
Note that the $m$-th entry of \,$A \ast \mathbf{x}$\, is\, $\sum_{n=1}^N (a_{m,n}\ast x_{n})$, where $x_{n}$ denotes the $n$-th entry of $\mathbf{x} \in \ell^2_{_N}(H)$. The usual properties of a discrete convolution can be found in Refs.~\cite{folland:95,fuhr:05}.
 
\medskip

The discrete {\em convolution system} with associated matrix $A=[a_{m,n}] \in \mathcal{M}_{_{M\times N}}\big(\ell^2(H)\big)$ given by
\begin{equation}
\label{syscon}
\begin{array}{rccl}
	&\mathcal{A}: \ell^2_{_N}(H)  &\longrightarrow &\ell^2_{_M}(H)\\
	  & \mathbf{x} &\longmapsto &  \mathcal{A}(\mathbf{x})=A\ast \mathbf{x}
    \end{array}
\end{equation}
is a well defined bounded operator if and only if $\widehat{A}\in \mathcal{M}_{_{M\times N}}\big(L^\infty(\widehat{H})\big)$, where $\widehat{A}(\xi):=\big[\widehat{a}_{m,n}(\xi)\big]$ denotes the {\em transfer matrix} of $A$ (see Refs.~\cite{garcia:19,gerardo:19}). We identify the matrix $\widehat{A}$ with entries in $L^\infty(\widehat{\Lambda})$ and the essentially bounded matrix-valued function $\widehat{A}(\xi)$, a.e. $\xi \in \widehat{\Lambda}$.

\medskip

Its adjoint operator $\mathcal{A}^*: \ell^2_M(H) \rightarrow \ell^2_N(H)$ is also a bounded convolution system with associated matrix $A^*=\big[a^*_{m,n}\big]^\top \in  \mathcal{M}_{_{N\times M}}\big(\ell^2(H)\big)$, where $a^*_{m,n}$ denotes the involution $a^*_{m,n}(h):=\overline{a_{m,n}(-h)}$, $h\in H$ (Refs.~\cite{garcia:19,gerardo:19}). Its transfer matrix is $\widehat{A^*}(\xi)$ is just the transpose conjugate of $\widehat{A}(\xi)$, i.e., 
$\widehat{A}(\xi)^*$, a.e. $\xi \in \widehat{H}$ (Refs.~\cite{garcia:19,gerardo:19}).  

\medskip

The bounded operator $\mathcal{A}$ is injective with a closed range if and only if the operator $\mathcal{A}^*\,\mathcal{A}$ is invertible; equivalently, the constant $\delta_A:=\einf_{\xi\in \widehat{H}} \det[\,\widehat{A}(\xi)^*\widehat{A}(\xi)]>0$ (see Ref.~\cite{garcia:19}).

\medskip

Let $\mathbf{a}^*_{m}$ denote the $m$-th column of the matrix $A^*$, then the $m$-th component of  $\mathcal{A}(\mathbf{x})$ is
\[
[A\ast \mathbf{x}]_{m}(h)=\sum_{n=1}^N (a_{m,n} \ast_{_H} x_n)(h) =\big\langle \mathbf{x}, T_{h}\,\mathbf{a}^*_{m} \big\rangle_{\ell^2_{_N}(H)}\,, \quad h\in H\,,
\]
where $T_h$ denotes the translation operator by $h\in H$ in $\ell^2_{_N}(H)$, i.e., for $\mathbf{a}\in \ell^2_{_N}(H)$,  $T_{h}\,\mathbf{a}(g)=\mathbf{a}(g-h)$, $g\in H$, and $\ast_{_H}$ is the convolution indexed by $H$.
In other words, the operator $\mathcal{A}$ is the {\em analysis operator} of the sequence $\big\{T_{h} \mathbf{a}^*_{m}\big\}_{h\in H;\, m=1,2,\dots,M}$ in $\ell^2_{N}(H)$. Since the sequence $\big\{T_{h}\,\mathbf{a}^*_{m}\big\}_{h\in H;\, m=1,2,\dots,M}$  is a frame for $\ell^2_{N}(H)$ if and only if its (bounded) analysis operator is injective with a closed range (see Ref.~\cite{ole:16}). Hence, the sequence $\big\{T_{h} \mathbf{a}^*_{m}\big\}_{h\in H;\, m=1,2,\dots,M}$ will be a frame for $\ell^2_{N}(H)$ if and only if the constant $\delta_A>0$.

\medskip

Concerning the duals of $\big\{T_{h}\,\mathbf{a}^*_{m}\big\}_{h\in H;\, m=1,2,\dots,M}$ having the same structure, consider two matrices $\widehat{A}\in \mathcal{M}_{_{M\times N}}(L^\infty(\widehat{H}))$ and $\widehat{B}\in \mathcal{M}_{_{N\times M}}(L^\infty(\widehat{H}))$, and let $\mathbf{b}_{m}$ denote the $m$-th column of the matrix $B$ associated to $\widehat{B}$. Then, the sequences $\big\{T_h\,\mathbf{a}^*_{m}\big\}_{h\in H;\, m=1,2,\dots,M}$ and $\big\{T_h\,\mathbf{b}_{m}\big\}_{h\in H;\, m=1,2,\dots,M}$ form a pair of dual frames for $\ell^2_{_N}(H)$ if and only if $\widehat{B}(\xi)\,\widehat{A}(\xi)=I_{_N}$,\,\, a.e. $\xi \in \widehat{H}$; equivalently, if and only if $\mathcal{B}\,\mathcal{A}=\mathcal{I}_{\ell^2_{N}(H)}$, i.e., the convolution system 
$\mathcal{B}$ is a left-inverse of the convolution system $\mathcal{A}$ (see Ref.~\cite{garcia:19}). Thus in $\ell^2_{_N}(H)$ we have the frame expansion
\[
\mathbf{x}=\sum_{m=1}^M \sum_{h\in H} \big\langle \mathbf{x}, T_h \mathbf{a}^*_{m} \big \rangle_{\ell^2_{_N}(H)}\, T_h \mathbf{b}_{m}\quad \text{for each $\mathbf{x} \in \ell^2_{_N}(H)$}\,.
\]
Finally to remind that the convolution system $\mathcal{A}$ in \eqref{syscon} is an isomorphism if and only if $M=N$ and $\einf_{\xi \in \widehat{H}}\big|\det \widehat{A}(\xi)\big|>0$ (see Refs.~\cite{garcia:19,gerardo:19}). Thus, for the case $M=N$ the sequence  $\big\{T_h\,\mathbf{a}_m^*\big\}_{h\in H;\,m=1, 2,\dots,N}$ is a Riesz basis for $\ell_{_N}^2(H)$. The square matrix $\widehat{A}(\xi)$ is invertible, a.e. $\xi \in \widehat{H}$, and from the columns of $\widehat{A}(\xi)^{-1}$ we get its dual Riesz basis $\big\{T_h\,\mathbf{b}_m\big\}_{h\in H;\,m=1, 2,\dots,N}$.

\section{The subspace of $L^2(G)$ where the sampling is carried out}
\label{section3}
Let $G \ni t \longmapsto U(t)\in \mathcal{U}(\mathcal{H})$ be a {\em unitary representation} of a LCA group $(G,+)$ on a separable Hilbert space $\mathcal{H}$. Recall that $\{U(t)\}_{t\in G}$ is a family of unitary operators in $\mathcal{H}$ satisfying:
$U(t)\,U(t')=U(t+t')$ for $t, t'\in G$; $U(0)=I_\mathcal{H}$; and $\big\langle U(t)\varphi, \phi\big\rangle_{\mathcal{H}}$ is a continuous function of $t$ for any $\varphi, \phi \in \mathcal{H}$. Note that $U(t)^{-1}=U(-t)$, and since $U(t)^*=U(t)^{-1}$ we have $U(t)^*=U(-t)$.

\medskip

Assume that for a fixed $\phi\in \mathcal{H}$ the family $\{U(t)\phi\}_{t\in G}$ is a continuous frame for the Hilbert space $\mathcal{H}$ with respect to $(G, \mu_G)$, where $\mu_G$ denotes the Haar measure on $G$. Let $H$ be a countable discrete subgroup of $G$. For a stable set of generators $\Phi=\{\varphi_1, \varphi_2, \dots, \varphi_N \} \subset \mathcal{H}$ we consider the $U$-invariant subspace in $\mathcal{H}$ generated by $\Phi$:
\[
\mathcal{H}_{\Phi}=\Big\{ \sum_{n=1}^N\sum_{h\in H} x_{n}(h) U(h)\varphi_{n} \, :\,  x_{n}\in \ell^2(H),\,\, n=1, 2, \dots, N\Big\}\,.
\]
We are assuming that the sequence $\{U(h)\varphi_{n}\}_{h\in H;\,n=1, 2,\dots,N}$ is a Riesz sequence in $\mathcal{H}$, i.e., a Riesz basis for $\mathcal{H}_\Phi$. Finally, we define the subspace in $L^2(G)$ given by
\[
\mathcal{H}_{U,\phi,\Phi}:=\Big\{ F_f:G\longrightarrow \mathbb{C}\,\,:\,\, F_f(t)=\big\langle f, U(t)\phi\big\rangle_{\mathcal{H}}\,, \,\, t\in G\,,\text{ where $f\in \mathcal{H}_{\Phi}$} \Big\}\,.
\]
Notice that the mapping $\mathcal{H}_{\Phi} \ni f \longmapsto F_f\in \mathcal{H}_{U,\phi,\Phi}$ is an isomorphism between Hilbert spaces, and the space 
$\mathcal{H}_{U,\phi,\Phi}$ is a RKHS of continuous and bounded functions in $L^2(G)$.

\medskip

Besides, the space $\mathcal{H}_{U,\phi,\Phi}$ is an $H$-shift-invariant subspace of $L^2(G)$; indeed, for each $t\in H$ we have that the function $F_f(s-t)=F_{U(t)f}(s)$, $s\in G$, and consequently it belongs to $\mathcal{H}_{U,\phi,\Phi}$.

\medskip

For instance, if $\phi=\sinc$ is the cardinal sine function then $\{U(t)\phi= \sinc (\cdot-t)\}_{t\in \mathbb{R}}$ is a continuous frame for $\mathcal{H}=PW_\pi$ with respect to $(\mathbb{R}, dx)$ and the space $\mathcal{H}_{U,\phi,\Phi}$ coincides with $PW_\pi$.

\medskip

From now on we will omit the subscript $f$ in the notation of $F_f$. 
The aim of this work is to obtain stable sampling results for the functions $F \in\mathcal{H}_{U,\phi,\Phi}$; for instance, the stable recovery of any $F\in \mathcal{H}_{U,\phi,\Phi}$ from the sequence of its samples $\{F(t)\}_{t\in H}$ taken at the countable subgroup $H$ of $G$ and/or other sequences of samples  $\{\mathcal{L}_mF(t)\}_{t\in H}$ introduced in  next section.
\subsection{Sampling data as a filtering process}
\label{section3-1}
Let $F$ be a function in $\mathcal{H}_{U,\phi,\Phi}$. For each $t\in H$, we have for the sample $F(t)$  the expression
\begin{equation}
\label{s1}
\begin{split}
F(t)&=\big\langle f,  U(t)\phi\big\rangle_{\mathcal{H}}=\big\langle \sum_{n=1}^N\sum_{h\in H} x_{n}(h) U(h)\varphi_{n}, U(t)\phi \big\rangle_\mathcal{H}\\
&=\sum_{n=1}^N\sum_{h\in H} x_{n}(h)\big\langle \varphi_n, U(t-h)\phi\big\rangle_\mathcal{H}=\sum_{n=1}^N\big(a_n\ast_{_H} x_n\big)(t),\quad t\in H\,,
\end{split}
\end{equation}
where $a_n(s):=\big\langle \varphi_n, U(s)\phi\big\rangle_\mathcal{H}$, $s\in H$, and $\ast_{_H}$ denotes the convolution in 
$\ell^2(H)$.

\medskip

Similarly, we can define generalized average samples as follows: Given $M$ fixed elements $\psi_1, \psi_2,\dots,\psi_M$ 
in $\mathcal{H}$, for any $F\in \mathcal{H}_{U,\phi,\Phi}$ we define the samples at $H$ as
\begin{equation}
\label{smean}
\mathcal{M}_mF(t):=\big\langle f, U(t)\psi_m\big\rangle_\mathcal{H},\quad t\in H\,, \text{ and $m=1, 2, \dots, M$}\,.
\end{equation}
Proceeding as in Eq.\eqref{s1}, for each $m=1, 2, \dots, M$ we get
\[
\mathcal{M}_mF(t)=\sum_{n=1}^N \big(a_{m,n}\ast_{_H} x_n\big)(t)\,, \quad t\in H\,,
\]
where $a_{m,n}(s):=\big\langle \varphi_n, U(s)\psi_m\big\rangle_\mathcal{H}$, $s\in H$. The above examples together with the results in Section \ref{section2-2} lead us to define in  $\mathcal{H}_{U,\phi,\Phi}$ a {\em generalized stable sampling procedure} at the subgroup $H$ as follows: 
\begin{dfn}
\label{sampdef}
Let $F(s)=\big\langle f, U(s)\phi\big\rangle_{\mathcal{H}}$, $s\in G$, be a function in $\mathcal{H}_{U,\phi,\Phi}$ and suppose that $f=\sum_{n=1}^N\sum_{h\in H} x_{n}(h) U(h)\varphi_{n}$ in $\mathcal{H}_{\Phi}$. A generalized stable sampling procedure $\boldsymbol{\mathcal{L}}_{_A}$ at $H$ in $\mathcal{H}_{U,\phi,\Phi}$ is defined for each  $F\in \mathcal{H}_{U,\phi,\Phi}$ by 
\begin{equation}
\label{s2}
\boldsymbol{\mathcal{L}}_{_A}F(t):=\big(A\ast_{_H} \mathbf{x}\big)(t), \quad t\in H\,,
\end{equation}
where $\boldsymbol{\mathcal{L}}_{_A}F(t):=\big(\mathcal{L}_1F(t), \mathcal{L}_2 F(t), \dots, \mathcal{L}_M F(t)\big)^\top$, $\mathbf{x}(t)=(x_1(t), x_2(t), \dots, x_N(t))^\top \in \ell^2_{_N}(H)$, and $A$ denotes a matrix $[a_{m,n}]\in  \mathcal{M}_{_{M\times N}}\big(\ell^2(H)\big)$ such that:
\begin{enumerate}
\item Its tranfer matrix $\widehat{A}\in \mathcal{M}_{_{M\times N}}\big(L^\infty(\widehat{H})\big)$, and 
\item the constant  $\delta_A:=\einf_{\xi\in \widehat{H}} \det[\widehat{A}(\xi)^*\widehat{A}(\xi)]>0$\,.
\end{enumerate}
\end{dfn}
Definition \ref{sampdef} is, of course, equivalent to the classical one  that states the existence of two positive constants $0< c \le C$ such that
\[
c\|F\|^2 \le \sum_{m=1}^M \sum_{t\in H} |\mathcal{L}_m F(t)|^2 \le C \|F\|^2 \quad \text{for any $F\in \mathcal{H}_{U,\phi,\Phi}$}\,.
\]
See Notes \ref{note2} and \ref{note8} in Section \ref{section4}.

\medskip

The definition of stable sampling as stated above shows, in a explicit way, the relationship between the stable samples and their associated sampling formulas. Indeed, as it will be proved in Theorem \ref{sampteo} (see Section \ref{section4}), once a generalized stable sampling procedure $\boldsymbol{\mathcal{L}}_{_A}$ is given in 
$\mathcal{H}_{U,\phi,\Phi}$, there exists a method  to obtain the associated stable sampling formulas of the form:
\[
F(s)=\sum_{m=1}^M\sum_{t\in H} \mathcal{L}_m F(t)\,S_m(s-t)\,, \quad s\in G\,,
\]
for every function $F\in \mathcal{H}_{U,\phi,\Phi}$. Namely:
\begin{itemize}
\item Compute a matrix $\widehat{B}(\xi)\in \mathcal{M}_{_{N\times M}}\big(L^\infty(\widehat{H})\big)$ such that $\widehat{B}(\xi)\, \widehat{A}(\xi)=I_{_N}$, a.e. $\xi \in \widehat{H}$.
\item Compute the matrix $B \in  \mathcal{M}_{_{N\times M}}\big(\ell^2(H)\big)$ such that $\widehat{B}(\xi)$ is its transfer matrix.
\item Let $\mathbf{b}_m=\big(b_{1,m}, b_{2,m},\dots ,b_{N,m}\big)^\top$ be the $m$-th column of the matrix $B$, $m=1, 2,\dots,M$.
\item Finally, we obtain the sampling functions as $S_m(s)=\big\langle \beta_m, U(s)\phi \big\rangle_\mathcal{H}$, $s\in G$,  where 
$\beta_m=\sum_{n=1}^N\sum_{h\in H} b_{n,m}(h) U(h)\varphi_{n}$, $m=1, 2, \dots, M$.
\end{itemize}
Condition $2.$ in Definition \ref{sampdef} implies necessarily that $M\geq N$, i.e., the number $N$ of stable generators used in the auxiliary space $\mathcal{H}_\Phi$ determines the minimal number $M$ of sequences of samples at $H$ that should be considered.
In next section we go into the details. 
\section{The main sampling result and consequences}
\label{section4}
In this section we state and prove a general sampling result for $\mathcal{H}_{U,\phi,\Phi}$ associated with a stable sampling procedure 
$\boldsymbol{\mathcal{L}}_{_A}$ at a subgroup $H$. We will see that other usual sampling results can be deduced from it.
\begin{teo}
\label{sampteo}
Assume that a generalized sampling procedure $\boldsymbol{\mathcal{L}}_{_A}$ at $H$, as in Definition \ref{sampdef},  has been defined in $\mathcal{H}_{U,\phi,\Phi}$. Then, there exist $M$ sampling functions $S_m$, $m=1, 2,\dots, M$, in $\mathcal{H}_{U,\phi,\Phi}$ such that the sequence $\{S_m(\cdot-t)\}_{t\in H;\,m=1, 2,\dots,M}$ is a frame for $\mathcal{H}_{U,\phi,\Phi}$, and the sampling expansion
\begin{equation}
\label{samf1}
F(s)=\sum_{m=1}^M\sum_{t\in H} \mathcal{L}_m F(t)\,S_m(s-t)\,, \quad s\in G\,,
\end{equation}
holds for every $F\in \mathcal{H}_{U,\phi,\Phi}$. The pointwise convergence of the above series is absolute in $G$ and uniform on $G$. It also converges in the $L^2(G)$-norm sense. 
\end{teo}
\begin{proof}
Consider  $F(s)=\big\langle f, U(s)\phi\big\rangle_{\mathcal{H}}$, $s\in G$, a function in $\mathcal{H}_{U,\phi,\Phi}$ and suppose that for the corresponding $f$ in $\mathcal{H}_{\Phi}$ we have the expansion $f=\sum_{n=1}^N\sum_{h\in H} x_{n}(h) U(h)\varphi_{n}$. According to the results in Section \ref{section2-2}, since the matrix $\widehat{A}$ has entries in $L^\infty(\widehat{H})$ and
$\delta_A>0$, the sequence $\big\{T_t\,\mathbf{a}_m^*\big\}_{t\in H;\,m=1, 2,\dots,M}$ is a frame for $\ell_{_N}^2(H)$ where 
$\mathbf{a}_m^*=(a_{m,1}^*, a_{m,2}^*, \dots, a_{m,N}^*)^\top$ in $\ell_{_N}^2(H)$ denotes the $m$-th column of the matrix $A^*=[a_{m,n}^*]^\top \in \mathcal{M}_{_{N\times M}}\big(\ell^2(H)\big)$ whose entries are the involutions $a_{m,n}^*(h)=\overline{a_{m,n}(-h)}$, $h\in H$. Moreover, $\mathcal{L}_{m}F(t)=\big\langle \mathbf{x}, T_t\, \mathbf{a}^*_{m} \big \rangle_{\ell^2_{_N}(H)}$, $t\in H$ and $m=1, 2, \dots, M$.

Furthermore, there exists a matrix $\widehat{B}(\xi)\in \mathcal{M}_{_{N\times M}}\big(L^\infty(\widehat{H})\big)$ such that $\widehat{B}(\xi)\, \widehat{A}(\xi)=I_{_N}$, a.e. $\xi \in \widehat{H}$; it suffices to take $\widehat{B}(\xi)=\widehat{A}(\xi)^\dag:=\big[\widehat{A}(\xi)^*\widehat{A}(\xi)\big]^{-1}\widehat{A}(\xi)^*$, the Moore-Penrose pseudo-inverse of  $\widehat{A}(\xi)$. Besides, the sequence $\big\{T_t\,\mathbf{b}_m\big\}_{t\in H;\,m=1, 2,\dots,M}$ in $\ell_{_N}^2(H)$ is a dual frame of $\big\{T_t\,\mathbf{a}_m^*\big\}_{t\in H;\,m=1, 2,\dots,M}$, where 
$\mathbf{b}_m$ is the $m$-th column of the matrix $B \in  \mathcal{M}_{_{N\times M}}\big(\ell^2(H)\big)$ whose transfer matrix is $\widehat{B}(\xi)$. As a consequence, for $\mathbf{x}=(x_1, x_2, \dots, x_N)^\top \in \ell^2_{_N}(H)$ associated to $f \in\mathcal{H}_{\Phi}$ we have
\begin{equation}
\label{l2conv}
\mathbf{x}=\sum_{m=1}^M \sum_{t\in H} \big\langle \mathbf{x}, T_t\,\mathbf{a}^*_{m} \big \rangle_{\ell^2_{_N}(H)}\, T_t\,\mathbf{b}_{m}=\sum_{m=1}^M\sum_{t\in H} \mathcal{L}_{m}F(t)\,T_t\,\mathbf{b}_{m} \quad \text{in $\ell^2_{_N}(H)$}\,.
\end{equation}
Now consider the isomorphism $\mathcal{T}_\Phi$ defined as
\[
\begin{array}[c]{ccll}
 & \mathcal{T}_\Phi:\ell_{_N}^2(H) & \longrightarrow & \mathcal{H}_\Phi\\
       & \mathbf{x} & \longmapsto & f=\sum_{n=1}^N\sum_{h\in H} x_{n}(h) U(h)\varphi_{n}\,,
\end{array}
\]
which satisfies the shifting property $\mathcal{T}_\Phi (T_t\,\mathbf{b})=U(t)(\mathcal{T}_\Phi \mathbf{b})$, $t\in H$ and 
$\mathbf{b} \in \ell_{_N}^2(H)$. 
Applying the isomorphism $\mathcal{T}_\Phi$ in Eq.~\eqref{l2conv}, we get 
\[
f=\sum_{m=1}^M\sum_{t\in H} \mathcal{L}_{m}F(t) U(t)(\mathcal{T}_\Phi \mathbf{b}_m)=\sum_{m=1}^M\sum_{t\in H} \mathcal{L}_{m}F(t) U(t)\beta_m\,, \quad \text{in $\mathcal{H}_\Phi$}\,, 
\]
where $\beta_m=\mathcal{T}_\Phi \mathbf{b}_m=\sum_{n=1}^N\sum_{h\in H} b_{n,m}(h) U(h)\varphi_{n} \in \mathcal{H}_\Phi$. Finally, for $F \in \mathcal{H}_{U,\phi,\Phi}$ we obtain
\begin{equation}
\label{sampf2}
\begin{split}
F(s)&=\big\langle \sum_{m=1}^M\sum_{t\in H} \mathcal{L}_{m}F(t) U(t)\beta_m, U(s)\phi \big\rangle 
=\sum_{m=1}^M\sum_{t\in H} \mathcal{L}_m F(t) \big\langle \beta_m, U(s-t)\phi \big\rangle_\mathcal{H}\\
&=\sum_{m=1}^M\sum_{t\in H} \mathcal{L}_m F(t)\, S_m(s-t)\,,\quad s\in G\,,
\end{split}
\end{equation}
where $S_m(s)=\big\langle \beta_m, U(s)\phi \big\rangle_\mathcal{H}$, \, $s\in G$, and $m=1, 2, \dots, M$. The pointwise convergence in \eqref{sampf2} is absolute due to the unconditional character  of  a frame expansion. It is uniform on $G$ due to the inequality $|F(s)| \le \|f\| \, \|U(s)\phi\| \le \|f\|\,\|\phi\|$ for all $s\in G$. The composition of isomorphisms
\begin{equation}
\label{isos}
\begin{array}[c]{cccccc}
 & \ell_{_N}^2(H) & \longrightarrow & \mathcal{H}_\Phi & \longrightarrow & \mathcal{H}_{U,\phi,\Phi}\\
       & \mathbf{x} & \longmapsto & f & \longmapsto  & F\,,
\end{array}
\end{equation}
shows that the $\ell^2_{_N}(H)$-convergence in expansion \eqref{l2conv} implies
the convergence of expansion in  \eqref{sampf2} in the $L^2(G)$-norm sense; besides, the sequence $\{S_m(\cdot-t)\}_{t\in H;\,m=1, 2,\dots,M}$ is a frame for $\mathcal{H}_{U,\phi,\Phi}$.   
\end{proof}
As a consequence of the above theorem we obtain a {\em Shannon-type} sampling formula for the space $\mathcal{H}_{U,\phi,\Phi}$:

\begin{cor}
\label{cor2}
In order to recover any $F\in \mathcal{H}_{U,\phi,\Phi}$ from its samples $\{F(t)\}_{t\in H}$, at the subgroup $H$, necessarily  $N=1$. Under conditions in Definition \ref{sampdef}, i.e., 
$\widehat{a}\in L^\infty(\widehat{H})$ and $\einf_{\xi \in \widehat{H}}|\, \widehat{a}(\xi)|>0$, where $a(s)=\big\langle \varphi, U(s)\phi \big\rangle_{\mathcal{H}}$, $s\in H$, there exists a unique sampling function $S_a\in \mathcal{H}_{U,\phi,\Phi}$ such that the sampling expansion
\begin{equation}
\label{samf2}
F(s)=\sum_{t\in H} F(t)\,S_a(s-t)\,, \quad s\in G\,,
\end{equation}
holds in $\mathcal{H}_{U,\phi,\Phi}$. The sequence $\{S_a(\cdot-t)\}_{t\in H}$ is a Riesz basis for $\mathcal{H}_{U,\phi,\Phi}$.
\end{cor}
\begin{proof}
In this scalar case, there exists a unique $\widehat{b}(\xi)\in L^\infty(\widehat{H})$ such that $\widehat{b}(\xi)\, \widehat{a}(\xi)=1$, a.e. $\xi \in \widehat{H}$. As a consequence, the associated sampling function $S_a(s)=\langle \beta, U(s)\phi \rangle_\mathcal{H}$, $s\in G$, where $\beta=\sum_{h\in H} b(h)\,U(h)\varphi$ belongs to $\mathcal{H}_\varphi$ and it is unique. 
\end{proof}
In case $N>1$ we must add $M-1$ sequences  $\{\mathcal{L}_m F(t)\}_{t\in H}$ of samples, $m=2,\dots,M$, with $M\geq N$, to the sequence $\{F(t)\}_{t\in H}$ such that the corresponding matrix $A=[a_{m,n}]$ in  $\mathcal{M}_{_{M\times N}}\big(\ell^2(H)\big)$ satisfies the conditions in Definition \ref{sampdef}. Thus, by using Theorem \ref{sampteo}, there exist $M$ sampling functions $S_m \in\mathcal{H}_{U,\phi,\Phi}$, $m=1, 2,\dots, M$,  such that $\{S_m(\cdot-t)\}_{t\in H;\,m=1, 2,\dots,M}$ is a frame for $\mathcal{H}_{U,\phi,\Phi}$, and the sampling expansion
\begin{equation}
\label{samf3}
F(s)=\sum_{t\in H} F(t)\,S_1(s-t)+\sum_{m=2}^M\sum_{t\in H} \mathcal{L}_m F(t)\,S_m(s-t)\,, \quad s\in G\,,
\end{equation}
holds for every $F\in \mathcal{H}_{U,\phi,\Phi}$.

\subsection{Sampling at a subgroup $R$ with finite index in $H$}
\label{section4-1}

The result in Theorem \ref{sampteo} can be easily modified in order to take just samples at a subgroup $R$ with finite index in $H$. Indeed, let  $R$ be a  subgroup of $H$ with finite index $L$. We fix a set $\{h_1, h_2, \dots,h_L\}$ of representatives of the cosets of $R$, i.e., the group $H$ can be decomposed as
\[
H= (h_1+R) \cup (h_2+R) \cup \dots \cup (h_{L}+R)\,\, \text{with}\,\, (h_l+R) \cap (h_{l'}+R)=\varnothing \,\,\text{for $l\neq l'$}\,.
\]
The space $\mathcal{H}_{\Phi}$ can be written as
\[
\begin{split}
\mathcal{H}_\Phi&=\Big\{\sum_{n=1}^N\sum_{h\in H} x_n(h)\,U(h)\varphi_n\, :\,  x_n \in \ell^2(H)\Big\}=
\Big\{\sum_{n=1}^N\sum_{l=1}^L \sum_{r\in R} x_n(h_l+r)\,U(h_l+r)\varphi_n\Big\}\\
&=\Big\{\sum_{n=1}^N\sum_{l=1}^L \sum_{r\in R} x_{nl}(r)\, U(r)\varphi_{nl}\, :\,  x_{nl} \in \ell^2(R)\Big\}\,,
\end{split}
\]
with $x_{nl}(r):=x_n(h_l+r)$ and $\varphi_{nl}:=U(h_l)\varphi_n$, where the new index $nl$ goes  from $11$ to $NL$.
Thus our subspace $\mathcal{H}_\Phi$ can be rewritten as $\mathcal{H}_{\widetilde{\Phi}}$ with $NL$ generators $\widetilde{\Phi}=\{\varphi_{nl}\}$  and coefficients $x_{nl}$ in $\ell^2(R)$. 

For instance, concerning the samples $\{F(r)\}_{r\in R}$ we have
\begin{equation}
\label{s3}
\begin{split}
F(r)&=\big\langle f,  U(r)\phi\big\rangle_{\mathcal{H}}=\big\langle \sum_{n=1}^N\sum_{l=1}^L\sum_{s\in R} x_{nl}(s) U(s)\varphi_{nl}, U(r)\phi \big\rangle_\mathcal{H}\\
&=\sum_{n=1}^N\sum_{l=1}^L\sum_{s\in R} x_{nl}(s)\big\langle \varphi_{nl}, U(r-s)\phi\big\rangle_\mathcal{H}=\sum_{n=1}^N\sum_{l=1}^L \big(a_{1,nl}\ast_{_R} x_{nl}\big)(r),\quad r\in R\,,
\end{split}
\end{equation}
where $a_{1,nl}(s)=\big\langle \varphi_{nl}, U(s)\phi\big\rangle_\mathcal{H}$, $s\in R$, and $\ast_{_R}$ denotes the convolution in $\ell^2(R)$. The new index runs as $nl=11, 12, \dots, 1L,\dots, N1, N2,\dots, NL$.  In general, we could consider a stable sampling procedure $\boldsymbol{\mathcal{L}}_{_A}$ at $R$ with associated matrix $A=[a_{m,nl}]\in  \mathcal{M}_{_{M\times NL}}\big(\ell^2(R)\big)$ as in Definition \ref{sampdef}. Thus we have:
\begin{cor}
\label{cor3}
Let  $A=[a_{m,nl}]\in  \mathcal{M}_{_{M\times NL}}\big(\ell^2(R)\big)$ be a matrix  associated to a stable sampling procedure 
$\boldsymbol{\mathcal{L}}_{_A}$ at $R$ as in Definition \ref{sampdef}. Then, there exist $M\geq NL$ sampling functions
 $S_m\in \mathcal{H}_{U,\phi,\Phi}$, $m=1, 2,\dots, M$, such that the sampling expansion
\begin{equation}
\label{samf4}
F(s)=\sum_{m=1}^M \sum_{r\in R} \mathcal{L}_m F(r)\,S_m(s-r)\,, \quad s\in G\,.
\end{equation}
holds in $\mathcal{H}_{U,\phi,\Phi}$. The sequence $\{S_m(\cdot-r)\}_{r\in R;\,m=1, 2,\dots,M}$ is a frame for 
$\mathcal{H}_{U,\phi,\Phi}$.
\end{cor}
\begin{proof}
Consider a matrix $B\in \mathcal{M}_{_{NL\times M}}\big(\ell^2(R)\big)$ such that such that $\widehat{B}(\xi)\in \mathcal{M}_{_{NL\times M}}\big(L^\infty(\widehat{R})\big)$ and $\widehat{B}(\xi)\, \widehat{A}(\xi)=I_{_{NL}}$, a.e. $\xi \in \widehat{R}$. The sampling functions are $S_m(s)=\langle \beta_m, U(s)\phi \rangle_\mathcal{H}$, $s\in G$, where $\beta_m=\mathcal{T}_{\widetilde{\Phi}} \mathbf{b}_m$ belongs to $\mathcal{H}_{\widetilde{\Phi}}$, and $\mathbf{b}_m\in \ell^2_{NL}(R)$ is the $m$-th column of the matrix $B$ with transfer matrix $\widehat{B}$, $m=1, 2,\dots, M$. 
\end{proof}
\subsection{Additional notes and remarks}
\label{section4-2}
Next we include some comments and remarks enlightening the above results: 
\begin{enumerate}
\item In Section \ref{section3} we are assuming that the sequence $\{U(h)\varphi_{n}\}_{h\in H;\,n=1, 2,\dots,N}$ is a Riesz sequence in $\mathcal{H}$, i.e., a Riesz basis for $\mathcal{H}_\Phi$.  Necessary and sufficient conditions for sequences having this unitary structure can be found in Refs.~\cite{aldroubi:96,barbieri:15, cabrelli:10,lee:93,jia:91,gerardo:19}. 

\item \label{note2} For any stable sampling procedure $\boldsymbol{\mathcal{L}}_{_A}$ at $H$ as in Definition \ref{sampdef}  there exist positive constants $0< c \le C$ such that
\begin{equation}
\label{stable}
c\|F\|^2 \le \sum_{m=1}^M \sum_{t\in H} |\mathcal{L}_m F(t)|^2 \le C \|F\|^2 \quad \text{for every $F\in \mathcal{H}_{U,\phi,\Phi}$}\,.
\end{equation}
Indeed, it follows since the sequence  $\big\{T_t\,\mathbf{a}_m^*\big\}_{t\in H;\,m=1, 2,\dots,M}$ is a frame for $\ell_{_N}^2(H)$ and from the isomorphim in Eq.~\eqref{isos}.

\item In case $M=N$ in Theorem \ref{sampteo},  the sequence $\big\{T_t\,\mathbf{a}_m^*\big\}_{t\in H;\,m=1, 2,\dots,N}$ is a Riesz basis for $\ell_{_N}^2(H)$ and the square matrix $\widehat{A}(\xi)$ is invertible; proceeding as in the proof of Theorem \ref{sampteo} its inverse gives the dual Riesz basis $\big\{T_t\,\mathbf{b}_m\big\}_{t\in H;\,m=1, 2,\dots,N}$. As a consequence, the sequence $\{S_n(\cdot-t)\}_{t\in H;\,n=1, 2,\dots,N}$ is a Riesz basis for $\mathcal{H}_{U,\phi,\Phi}$. The uniqueness of the coefficients in a Riesz basis expansion gives the {\em interpolation property}:
\[ 
\mathcal{L}_n S_{n'}(t-t')=\delta_{n,n'}\delta_{t,t'}\,,\,\, \text{ where $t,t'\in H$ and $n, n'=1,2,\dots,N$}\,. 
\]

\item The use of multiple generators for the auxiliary space $\mathcal{H}_\Phi$ might have  advantages against the single generator case. For instance, in Corollary \ref{cor2} the conditions of stable sampling in Definition \ref{sampdef} for a single generator $\varphi$ are: $\widehat{a}\in L^\infty(\widehat{H})$ and $\einf_{\xi \in \widehat{H}}|\widehat{a}(\xi)|>0$, where $a(s)=\big\langle \varphi, U(s)\phi \big\rangle_{\mathcal{H}}$, $s\in H$. A way to overcome the restrictive second condition is to consider a subgroup $R$ of $H$ with finite index $L>1$, and then consider $M\geq L$ sequences of samples taken at the subgroup $R$.

\item In Theorem \ref{sampteo},  whenever $M>N$,  there exist infinite sampling functions $S_m$ coming from the different dual frames $\big\{T_t\,\mathbf{b}_m\big\}_{t\in H;\,m=1, 2,\dots,M}$ of $\big\{T_t\,\mathbf{a}_m^*\big\}_{t\in H;\,m=1, 2,\dots,M}$. All these duals come from the left-inverses $\widehat{B}(\xi)$ of $\widehat{A}(\xi)$ which are obtained, from the Moore-Penrose pseudo-inverse $\widehat{A}(\xi)^\dag=\big[\widehat{A}(\xi)^*\widehat{A}(\xi)\big]^{-1}\widehat{A}(\xi)^*$, by means of the $N\times M$ matrices 
\[
\widehat{B}(\xi):=\widehat{A}(\xi)^\dag+C(\xi)\big[I_M-\widehat{A}(\xi)\widehat{A}(\xi)^\dag\big]\,, 
\]
where $C(\xi)$ denotes any $N\times M$ matrix with entries in $L^\infty(\widehat{G})$. Indeed, it is straightforward to check that any matrix having this form is a left-inverse of $\widehat{A}(\xi)$. Moreover, any left-inverse 
$\widehat{B}(\xi)$ of $\widehat{A}(\xi)$ belongs to the above family; it suffices to take $C(\xi)=\widehat{B}(\xi)$.

\item The sequence $\big\{T_t\,\mathbf{a}_m^*\big\}_{t\in H;\,m=1, 2,\dots,M}$ is a Bessel sequence in $\ell_{_N}^2(H)$ if and only if the convolution system in \eqref{syscon} is bounded, i.e., if and only if the transfer matrix $\widehat{A}(\xi)$ belongs to $\mathcal{M}_{_{M\times N}}\big(L^\infty(\widehat{H})\big)$.
Moreover, having in mind the equivalence between the spectral and Frobenius norms for matrices (see Ref.~\cite{horn:99}), it is equivalent to the condition 
\[
\beta_A:= \esup_{\xi\in \widehat{H}} \lambda_{\text{max}}[\widehat{A}(\xi)^*\widehat{A}(\xi)]<+\infty\,,
\] 
where $\lambda_{\text{max}}$ denotes the largest eigenvalue of the positive semidefinite matrix 
$\widehat{A}(\xi)^*\widehat{A}(\xi)$.

\item \label{note8} Under the hypothesis $\widehat{A}\in \mathcal{M}_{_{M\times N}}\big(L^\infty(\widehat{H})\big)$, the condition 
$\delta_A>0$ in Definition \ref{sampdef} is equivalent to the condition $\alpha_A:=\einf_{\xi\in \widehat{H}} \lambda_{\text{min}}[\widehat{A}(\xi)^*\widehat{A}(\xi)]>0$, where $\lambda_{\text{min}}$ denotes the smallest eigenvalue of the positive semidefinite matrix $\widehat{A}(\xi)^*\widehat{A}(\xi)$. Indeed, it comes from the inequalities:
\[
\alpha_A^N \le \delta_A \le \alpha_A\, \beta_A^{N-1}\,,
\]
where $\beta_A$ was introduced in the above note.

\item Assuming that $\widehat{A}\in \mathcal{M}_{_{M\times N}}\big(L^\infty(\widehat{H})\big)$, it is worth to mention that the condition $\delta_A>0$ is also necessary  for the existence of a frame expansion \eqref{samf1} as those in Theorem \ref{sampteo}. Indeed, suppose that, for each $F\in \mathcal{H}_{U,\phi,\Phi}$, a frame expansion $F=\sum_{m=1}^M\sum_{t\in H} \mathcal{L}_m F(t)\,S_m(\cdot-t)$ holds, and let $\mathcal{T}$ denote the isomorphism in Eq.\eqref{isos}. Then we have 
\[
\mathbf{x}=\sum_{m=1}^M \sum_{t\in H} \big\langle \mathbf{x}, T_t\,\mathbf{a}^*_{m} \big \rangle_{\ell^2_{_N}(H)}\, \mathcal{T}^{-1} S_m(\cdot-t)\quad \text{ for all $\mathbf{x} \in \ell^2_{_N}(H)$}\,.
\]
Since $\{T_t\,\mathbf{a}^*_{m}\}_{t\in H;\,m=1, 2,\dots,M}$ is a Bessel sequence, the sequences $\{T_t\,\mathbf{a}^*_{m}\}_{t\in H;\,m=1, 2,\dots,M}$ and $\{\mathcal{T}^{-1} S_m(\cdot-t)\}_{t\in H;\,m=1, 2,\dots,M}$ form a pair of dual frames in $\ell^2_{_N}(H)$ (see Ref.~\cite{ole:16}). In particular, according to a result in Section \ref{section2-2}, since $\{T_t\,\mathbf{a}^*_{m}\}_{t\in H;\,m=1, 2,\dots,M}$ is a frame for $\ell^2_{_N}(H)$ we get $\delta_A>0$.

\item Whenever the entries $a_{m,n}$ of the matrix $A$ belong to $\ell^1(H)$, their Fourier transforms $\widehat{a}_{m,n}$ are continuous and consequently  they belong to $L^\infty(\widehat{H})$. In this case the condition $\delta_A>0$ is equivalent to $\det [\widehat{A}(\xi)^*\widehat{A}(\xi)]\neq 0$ for all $\xi \in \widehat{H}$.

\item For the samples given in Eqs.~\eqref{s1}-\eqref{smean}, the entries $a_{m,n}$ of the matrix $A$ depend on the unitary representation $U(t)$ of the group $G$, $\psi$ and $\Phi$. For a general stable sampling procedure $\boldsymbol{\mathcal{L}}_{_A}$ at $H$, by changing $\psi$ and $\Phi$ we could recover a function $F$, in different spaces $\mathcal{H}_{U,\phi,\Phi}$, from the same sequence of samples.

\item We can relax the initial assumptions in this section by assuming that $\{U(t)\phi\}_{t\in G}$ is just a complete  Bessel family for $\mathcal{H}$ with respect to $(G, \mu_G)$. In this case the mapping $f\in \mathcal{H}_{\Phi} \longmapsto F_f\in \mathcal{H}_{U,\phi,\Phi}$ is injective and continuous but not necessarily an isomorphism (the subspace $\mathcal{H}_{U,\phi,\Phi}$ is not necessarily closed in $L^2(G)$). Under the hypotheses in Theorem \ref{sampteo}, the sampling formula \eqref{samf1} holds but the left-hand inequality in \eqref{stable} does not hold.
\end{enumerate}

\subsection{The case of a semi-direct product of groups}
\label{section4-3}
The case where the group $G$ is the {\em semi-direct product of two groups} can be easily reduced to the  situation described in Section \ref{section4} under appropriate conditions. Let $G=K\rtimes_\sigma H$ be the semi-direct product of the LCA  group $(K,+)$ and a not necessarily abelian group $(H, \cdot)$, where $\sigma$ denotes the action of the group $H$ on the group $K$, i.e., a homomorphism $\sigma: H \rightarrow Aut(K)$ mapping $h\mapsto \sigma_h$. The composition law in $G$ is 
$(k_1,h_1)\,(k_2,h_2):=(k_1+\sigma_{h_1}(k_2),h_1h_2)$ for $(k_1,h_1), \,(k_2,h_2) \in G$.  In general, the group $G=K\rtimes_\sigma H$ is not abelian. In case $\sigma_h\equiv Id_K$ for each $h\in H$ we recover the direct product group $G=K\times H$.

\medskip

Now we consider a subgroup $G'=K'\rtimes_\sigma H'$ where $(K',+)$ is a countable discrete group and $(H', \cdot)$ is a finite group of order $N$ (we will write $H'=\{h_1=1_H, h_2, \cdots, h_{_N}\}$), and such that $\sigma_h(K')=K'$ for each $h\in H'$.
Suppose that $(k,h)\longmapsto U(k,h)$ is a unitary representation of  the group $G=K\rtimes_\sigma H$ on a separable Hilbert space $\mathcal{H}$. For a fixed $\varphi \in \mathcal{H}$, the corresponding $\mathcal{H}_\varphi$ subspace can be written as
\begin{equation}
\label{newspace}
\begin{split}
\mathcal{H}_\varphi&=\Big\{\sum_{(k,h)\in G'} x(k,h)\, U(k,h)\varphi \,:\, \{x(k,h)\}\in \ell^2(G')\Big\} \\
&=\Big\{\sum_{n=1}^N \sum_{k\in K'} x(k,h_n)\, U(k,1_H)U(0_K,h_n)\varphi  \Big\} \\
&=\Big\{\sum_{n=1}^N \sum_{k\in K'}x_n(k)\,U(k,1_H)\varphi_n \,:\, \{x_n\}\in \ell^2(K') \Big\}\,,
\end{split}
\end{equation}
where $x_n(k):=x(k,h_n)$, $k\in K'$, and $\varphi_n:=U(0_K,h_n)\varphi $, $n=1, 2, \dots, N$. That is, $\mathcal{H}_\varphi\equiv\mathcal{H}_\Phi$ where $\Phi=\{ \varphi_1, \varphi_2, \dots, \varphi_N\}$ is a set of $N$ generators in $\mathcal{H}$. Assume that $\{U(s,t)\phi\}_{(s,t)\in G}$ is a continuous frame for $\mathcal{H}$ with respect to $(G, \mu_G)$, and consider the corresponding $\mathcal{H}_{U,\phi,\Phi}$ space. For any function $F(s,t)=\big\langle f, U(s,t)\phi\big\rangle_\mathcal{H}$, $(s,t)\in G$, in $\mathcal{H}_{U,\phi,\Phi}$ we define a sampling procedure at $K'$ by
\[
\boldsymbol{\mathcal{L}}_{_A}F(k)=\big(\mathcal{L}_1F(k), \mathcal{L}_2F(k), \dots, \mathcal{L}_{_M}F(k) \big)^\top:=\big(A\ast_{_{K'}} \mathbf{x}\big)(k), \quad k\in K'\,,
\]
where $A=[a_{m,n}]\in  \mathcal{M}_{_{M\times N}}\big(\ell^2(K')\big)$ and $\mathbf{x}=(x_1, x_2, \dots, x_{_N})^\top \in \ell^2_{_N}(K')$. Under conditions in Definition \ref{sampdef}, there exist $M$ elements $\beta_m\in \mathcal{H}_\Phi$, $m=1, 2, \dots, M$, such that
\[
f=\sum_{m=1}^M \sum_{k\in K'} \mathcal{L}_mF(k)\, U(k,1_H)\,\beta_m \quad \text{in $\mathcal{H}_\Phi$}\,.
\]
Remind that the elements $\beta_m\in \mathcal{H}_\Phi$, $m=1, 2, \dots, M$, are obtained from a left-inverse 
$\widehat{B}(\xi)$ of $\widehat{A}(\xi)$, a.e. $\xi \in \widehat{K'}$, as in Theorem \ref{sampteo}.
Hence, for each $F\in \mathcal{H}_{U,\phi,\Phi}$ we have
\begin{equation}
\label{samf5}
\begin{split}
F(s,t)&=\big\langle f, U(s,t)\phi\big\rangle_\mathcal{H}=\sum_{m=1}^M \sum_{k\in K'}\mathcal{L}_mF(k) \big\langle U(k,1_H)\,\beta_m,U(s,t)\,\phi \big\rangle_\mathcal{H}\\
&=\sum_{m=1}^M \sum_{k\in K'}\mathcal{L}_mF(k) \big\langle \beta_m, U\big[(k,1_H)^{-1}(s,t)\big]\,\phi \big\rangle_\mathcal{H}\\
&=\sum_{m=1}^M \sum_{k\in K'}\mathcal{L}_mF(k)\, S_m(s-k,t)\,, \quad (s,t)\in G\,,
\end{split}
\end{equation}
where $S_m(s,t)=\big\langle \beta_m, U(s,t)\,\phi \big\rangle_\mathcal{H}$, $(s,t)\in G$, $m=1, 2, \dots, M$. Notice that 
$(k,1_H)^{-1}(s,t)=(-k,1_H)(s,t)=(s-k,t)$, for $s\in K$, $k\in K'$ and $t\in H$.
\subsubsection{Euclidean motion group and crystallographic subgroups}
An example of the above setting  is given by {\em crystallographic groups} as subgroups of the {\em Euclidean motion group} $E(d)$. This group is the semi-direct product $\mathbb{R}^d \rtimes_{\sigma} O(d)$ corresponding to the homomorphism $\sigma : O(d) \rightarrow Aut(\mathbb{R}^d)$ given by $\sigma_{\gamma}(x) = \gamma x$, where $\gamma \in O(d)$ and $x\in \mathbb{R}^d$;  $O(d)$ denotes the orthogonal group of order $d$. The composition law on $E(d) = \mathbb{R}^d \rtimes_{\sigma} O(d)$ reads $(x, \gamma) \cdot  (x', \gamma') = (x + \gamma x', \gamma \gamma')$.

\medskip

The subgroup $G'$ would be the crystallographic group $\mathcal{C}_{P,\Gamma}:=P\mathbb{Z}^d \rtimes_\sigma \Gamma$ where $P$ is a non-singular $d\times d$ matrix and $\Gamma$ is a finite subgroup of $O(d)$ of order $N$ such that  
$\gamma(P\mathbb{Z}^d)=P\mathbb{Z}^d$ for each $\gamma \in \Gamma$. We will denote $\{\gamma_1=I, \gamma_2, \dots, \gamma_N\}$ the elements of the group $\Gamma$.
In this example we consider the {\em quasi regular representation} (see Ref.~\cite{barbieri:15}) on $L^2(\mathbb{R}^d)$:
\[
U(s,\gamma)f(t)=f[\gamma^{\top}(t-s)]\,,\quad \text{$t, s\in \mathbb{R}^d $, $\gamma\in O(d)$ and $f\in L^2(\mathbb{R}^d)$}\,.
\]
Assume that $\phi \in L^2(\mathbb{R}^d)$ is a function such that the family $\{U(s,\gamma)\phi\}_{(s,\gamma)\in E(d)}$ is a continuous frame for a closed subspace $\mathcal{H}$ of $L^2(\mathbb{R}^d)$ (containing $\mathcal{H}_\varphi$) with respect to $(E(d), ds\,d\mu(\gamma))$, where $d\mu(\gamma)$ denotes the left Haar measure on the group $O(d)$. For instance we could take $\phi$ a bandlimited function to a compact set $\Omega \subset \mathbb{R}^d$ and consider $\mathcal{H}:=PW_\Omega$; the details are similar to those in Ref.~\cite{garcia:20}. See Ref.~\cite{barbieri:15} for the details on the left Haar measure in semi-direct products of groups.

\medskip

For any function $F(s,\gamma)=\big\langle f(\cdot), \phi(\gamma^\top(\cdot-s))\big\rangle_{L^2(\mathbb{R}^d)}$, $(s,\gamma) \in E(d)$, with $f\in \mathcal{H}_\Phi$, we consider a stable sampling procedure $\boldsymbol{\mathcal{L}}_{_A}F(p):=(A\ast_{_{P\mathbb{Z}^d}} \mathbf{x})(p)$, $p\in P\mathbb{Z}^d$, defined at the lattice $P\mathbb{Z}^d$ as in Definition \ref{sampdef}. Remind that in this particular example the space defined in \eqref{newspace} is: 
\[
\mathcal{H}_\Phi=\Big\{\sum_{n=1}^N \sum_{p\in P\mathbb{Z}^d} x_n(p)\,\varphi_n(t-p) \,:\, \{x_n\}\in \ell^2(P\mathbb{Z}^d) \Big\}\,, 
\]
with $N$ generators $\varphi_n(t)=\varphi(\gamma_n^\top t)$, $n=1, 2, \dots, N$, in $L^2(\mathbb{R}^d)$.
The corresponding sampling formula \eqref{samf5} reads:
\[
F(s, \gamma)=\sum_{m=1}^M \sum_{p\in P\mathbb{Z}^d}\mathcal{L}_mF(p)\, S_m(s-p,\gamma)\,, \quad (s, \gamma)\in E(d)\,,
\]
where $S_m(s, \gamma)=\big\langle \beta_m(\cdot), \phi[\gamma^\top(\cdot-s)]\big\rangle_{L^2(\mathbb{R}^d)}$, 
$(s, \gamma)\in E(d)$, for some functions $\beta_m \in \mathcal{H}_\Phi$, $m=1, 2, \dots,M$, obtained from a left-inverse of the matrix $\widehat{A}(\xi)$ as in Theorem \ref{sampteo}.

\medskip

Whenever $f=\sum_{n=1}^N \sum_{p\in P\mathbb{Z}^d} x_n(p)\,\varphi_n(t-p)\in \mathcal{H}_\Phi$, for the pointwise samples $F(p, I)=\big\langle f(\cdot), \phi(\cdot-p)\big\rangle_{L^2(\mathbb{R}^d)}$, $p\in P\mathbb{Z}^d$ we have the expression $F(p, I)=\sum_{n=1}^N (a_{1,n} \ast x_n)(p)$, $p\in P\mathbb{Z}^d$, where 
\[
a_{1,n}(k)=\big\langle \varphi_n(\cdot), \phi(\cdot-k)\big\rangle_{L^2(\mathbb{R}^d)}=\big\langle \varphi(t), \phi(\gamma_n t-k)\big\rangle_{L^2(\mathbb{R}^d)}\,, \,\, \text{$k\in P\mathbb{Z}^d$ and $n=1, 2, \dots, N$}. 
\]
\subsection{Some final comments}
\label{section4-4}
The theory obtained in this work relies on the use of an LCA group $G$. We have considered non-abelian groups which are semi-direct product of groups; the case treated here can be reduced to the abelian case by increasing the number of generators in the auxiliary space. Formally, the general non-abelian case can be handled in the same way but necessarily it will need other additional mathematical tools (see, for instance, Refs.~\cite{barbieri:15,skret:20}). 

\bigskip

As an example, consider the (positive) {\em affine group} $G_+=\{ (a,b)\,:\, a>0\,, \, b\in \mathbb{R} \}$ with composition law $(a,b)\cdot(a',b')=(aa', b+ab')$, and its unitary representation $(a,b)\mapsto U(a,b)$ on $L^2(\mathbb{R})$ given by 
\[
\big[ U(a,b)f\big](t)=\frac{1}{\sqrt{a}}f\Big(\frac{t-b}{a} \Big)\,, \,\, \text{$t\in \mathbb{R}$, \,\,where $f\in L^2(\mathbb{R})$}\,. 
\]
The non-abelian group $G_+$ is non-unimodular with left Haar measure $d\mu_{_l}=\frac{dadb}{a^2}$.

Let  $\phi$ be a function in $L^2(\mathbb{R})$ such that $\{U(a,b)\phi\}_{(a,b)\in G_+}$ is a continuous frame for 
$ L^2(\mathbb{R})$, and let $\{\psi_{m,n}\}_{m,n\in \mathbb{Z}}$ be an orthonormal basis of wavelets for $L^2(\mathbb{R})$ where we use the notation  $\psi_{m,n}(t)=2^{-m/2}\psi \big(\frac{t-n}{2^m} \big)$, $m,n\in  \mathbb{Z}$.

\medskip

Now we sample any function $F(a,b)=\big\langle f, U(a,b)\phi \big\rangle_{L^2(\mathbb{R})}$, $(a,b)\in G_+$, where $f\in L^2(\mathbb{R})$, at the subspace $\Gamma:=\{(2^m,n)\,:\, m,n \in \mathbb{Z}\}$ of $G_+$. The functions $F$ defined above form a RKHS contained in $L^2(\mathbb{R}^+\times \mathbb{R}\,; \frac{dadb}{a^2})$.

\medskip

As in Section \ref{section3-1}, the samples $\{F(2^m,n)\}_{m,n\in  \mathbb{Z}}$ can be expressed as a discrete convolution in $\ell^2(\Gamma)$. A straightforward computation gives:
\[
F(2^m,n)=\sum_{p,q} \big\langle \psi, U\big[(2^p,q)^{-1}(2^m,n) \big]\phi \big\rangle_{L^2(\mathbb{R})}\, \big\langle f, \psi_{p,q} \big\rangle_{L^2(\mathbb{R})}=\big(\mathbf{a}\ast_\Gamma \mathbf{b}\big)(2^m,n)\,,\,\, (2^m,n)\in \Gamma \,,
\]
where $\mathbf{a}(m,n)=\big\langle \psi, U\big[(m,n) \big]\phi \big\rangle_{L^2(\mathbb{R})}$ and $\mathbf{b}(m,n)=\big\langle f, \psi_{m,n} \big\rangle_{L^2(\mathbb{R})}$, $(2^m,n)\in \Gamma$. The mathematical techniques used in Section \ref{section2-2} do not work for the non-abelian group $G_+$, and other mathematical techniques are necessary (see, for instance, Refs.~\cite{barbieri:15,skret:20}).

\bigskip

Another related classical problem is the following: let $\mathcal{H}_k$ be a RKHS of continuous functions $f:\mathbb{R} \rightarrow \mathbb{C}$ contained in $L^2(\mathbb{R})$ with reproducing kernels $\{k_x\}_{k\in \mathbb{R}}$. Assume that there exists a Riesz basis for $\mathcal{H}_k$ having the form $\{\varphi_n(t-kN)\}_{k\in \mathbb{Z};\,n=1, 2,\dots,N}$, and we want to recover any function $f\in \mathcal{H}_k$ from the sequences of its samples $\{f(kN+r)\}_{k\in \mathbb{Z};\,r=0,1,\dots, N-1}$ in a stable way. For the sample $f(kN+r)=\big\langle f, k_{_{kN+r}} \big\rangle_{L^2(\mathbb{R})}$, $k\in  \mathbb{Z}$ and $r=0,1,\dots, N-1$, we have the expression:
\[
\begin{split}
f(kN+r)&=\big\langle f, k_{_{kN+r}} \big\rangle_{L^2(\mathbb{R})}=\Big\langle \sum_{n=1}^N \sum_{m\in \mathbb{Z}} x_n(mN)\, \varphi_n(\cdot-mN), k_{_{kN+r}} \Big\rangle_{L^2(\mathbb{R})}\\
&=\sum_{n=1}^N \sum_{m\in \mathbb{Z}} x_n(mN)\, \varphi_n(kN+r-mN)=
\sum_{n=1}^N \big(a_{r,n} \ast_{_{N\mathbb{Z}}} x_n\big)(kN)\,,
\end{split}
\]
where $a_{r,n}(mN)=\varphi_n(mN-r)$, $m\in \mathbb{Z}$, for $n=1, 2,\dots,N$. Under the conditions in Definition \ref{sampdef} for the $N\times N$ matrix $A=[a_{r,n}]$, there exist $N$  sampling functions $S_r\in \mathcal{H}_k$, $r=0,1,\dots, N-1$, such that the sampling formula 
\[
f(t)=\sum_{r=0}^{N-1}\sum_{k\in  \mathbb{Z}} f(kN+r)\, S_r(t-kN)\,, \quad t\in \mathbb{R}\,,
\]
holds in $\mathcal{H}_k$. The sequence $\{S_r(\cdot-kN)\}_{_{k\in \mathbb{Z};\,r=0,1,\dots, N-1}}$ is a Riesz basis for 
$\mathcal{H}_k$.

\bigskip

Finally to say that there is some affinity of the approach followed in this work with the topic of dynamical sampling (see, for instance, Ref.~\cite{aldroubi:20} and references therein). Indeed, from the correlation between a continuous frame $\{U(t)\phi\}_{t\in G}$ and an element $f$ in a suitable Hilbert space $\mathcal{H}_\Phi$ we obtain a function $F(t)$ in $L^2(G)$. In the case studied here, assuming that the space $\mathcal{H}_\Phi$ has a discrete unitary structure and under appropriate conditions, the function $F$ can be recovered, in a stable way, from a finite number of data sequences. An important difference is that dynamical sampling approach relies on a semigroup structure rather than on a group one.

\bigskip

\noindent{\bf Acknowledgments:}
The author thanks {\em Universidad Carlos III de Madrid} for granting him a sabbatical year in 2020-21.
This work has been supported by the grant MTM2017-84098-P from the Spanish {\em Ministerio de Econom\'{\i}a y Competitividad (MINECO)}.



\begin{thebibliography}{10}

\bibitem{aldroubi:96}
A.~Aldroubi.
\newblock Oblique proyections in atomic spaces.
\newblock {\em Proc. Amer. Math. Soc.}, 124:2051--2060 (1996).

\bibitem{aldroubi:05}
A.~Aldroubi, Q.~Sun and W.~S.~Tang.
\newblock Convolution, average sampling, and a Calderon resolution of the
identity for shift-invariant spaces.
\newblock {\em J. Fourier Anal. Appl.}, 11(2):215--244 (2005).

\bibitem{aldroubi:20}
A.~Aldroubi, K.~Gr\"{o}chenig, L.~Huang, P.~Jaming, I.~Kryshtal and J.~L.~Romero.
\newblock Sampling the flow of a bandlimited function.
\newblock arXiv:2004.14032v1[math.CA] (2020).

\bibitem{antoine:93}
S. T.~Ali, J. P.~Antoine and J. P.~Gazeau.
\newblock Continuous frames in Hilbert spaces.
\newblock {\em Ann. Physics}, 222:1--37 (1993).

\bibitem{barbieri:15}
D.~Barbieri, E.~Hern\'andez and J.~Parcet.
\newblock Riesz and frame systems generated by unitary actions of discrete groups.
\newblock {\em Appl. Comput. Harmon. Anal.}, 39(3): 369--399 (2015).

\bibitem{bhandari:12}
A.~Bhandari and A.~I.~Zayed.
\newblock Shift-invariant and sampling spaces associated with the fractional Fourier transform domain.
\newblock {\em IEEE Trans. Signal Process.}, 60(4):1627--1637 (2012).

\bibitem{cabrelli:10}
C.~Cabrelli and V.~Paternostro.
\newblock Shift-invariant spaces on LCA groups.
 \newblock {\em J. Funct. Anal.}, 258:2034--2059 (2010).

\bibitem{ole:16}
O.~Christensen.
\newblock {\em An {I}ntroduction to {F}rames and {R}iesz {B}ases}, 2nd ed.,
\newblock Birkh{\"a}user, Boston (2016).

\bibitem{hector:14}
H.~R. Fern\'andez-Morales, A.~G. Garc\'{\i}a, M.~A. Hern\'andez-Medina and M.~J. Mu\~noz-Bouzo.
\newblock Generalized sampling: from shift-invariant to  $U$-invariant spaces.
\newblock {\em Anal. Appl.}, 13(3):303--329 (2015).

\bibitem{fornasier:05}
M.~Fornasier and H.~Rauhut.
\newblock Continuous frames, function spaces, and the discretization problem.
\newblock {\em J. Fourier Anal. Appl.}, 11(3):245--287 (2005).

\bibitem{folland:95}
G.~B.~Folland.
\newblock {\em A Course in Abstract Harmonic Analysis},
\newblock CRC Press (1995).

\bibitem{fuhr:05}
H.~F\"uhr.
\newblock {\em Abstract Harmonic Analysis of Continuous Wavelet Transform}.
\newblock Springer (2005).

\bibitem{gabardo:03}
J. P.~Gabardo and D.~Han.
\newblock Frames associated with measurable spaces.
\newblock {\em Adv. Comp. Math.}, 18(3):127--147 (2003).

\bibitem{garcia:19}
A. G.~Garc\'{\i}a, M.A.~Hern\'andez-Medina  and G.~P\'erez-Villal\'on.
\newblock Convolution systems on discrete abelian groups as a unifying strategy in sampling theory.
\newblock {\em Results Math.}, 75:40 (2020).

\bibitem{garcia:20}
A. G.~Garc\'{\i}a and M. J.~Mu\~noz-Bouzo.
\newblock A note on continuous stable sampling.
\newblock {\em Adv. Oper. Theory}, 5(3):994--1013 (2020). 

\bibitem{lee:93}
T.~N.~Goodman, S.~L.~Lee and W.~S.Tang.
\newblock Wavelet bases for a set of commuting unitary operators. 
\newblock {\em Adv. Comput. Math.}, 1(1):109-126 (1993).

\bibitem{horn:99}
R.~A. Horn and C.~R. Johnson.
\newblock {\em Matrix Analysis}.
\newblock Cambridge University Press (1999).

\bibitem{jia:91}
R.~Q. Jia and C.~A. Micchelli.
\newblock Using the refinement equations for the construction of pre-waveles {II}: Powers of two.
\newblock In {\em Curves and Surfaces}. P. J. Laurent, Le M\'ehaut\'e, L. L. Schumaker, Eds.
Academic Press, Boston, pp. 209--246 (1991).

\bibitem{kang:11}
S.~Kang and K.~H. Kwon.
\newblock Generalized average sampling in shift-invariant spaces.
\newblock {\em J. Math. Anal. Appl.}, 377:70--78 (2011).

\bibitem{gerardo:19}
G.~P\'erez-Villal\'on.
\newblock Discrete convolution operators and  Riesz systems generated by actions of abelian groups.
\newblock  {\em Ann. Funct. Anal.}, 11:285--297 (2020).

\bibitem{pohl:12}
V.~Pohl and H.~Boche.
\newblock $U$-invariant sampling and reconstruction in atomic
spaces with multiple generators.
\newblock {\em IEEE Trans. Signal Process.}, 60(7):3506--3519 (2012).

\bibitem{rahimi:06}
A.~Rahimi, A.~Najati and Y. N.~Dehghan.
\newblock Continuous frames in Hilbert spaces.
\newblock {\em Methods Funct. Anal. Topology}, 12(2):170--182 (2006).

\bibitem{shang:07}
Z.~Shang, W.~Sun and X.~Zhou.
\newblock Vector sampling expansions in shift-invariant subspaces.
\newblock {\em J. Math. Anal. Appl.}, 325:898--919 (2007).

\bibitem{skret:20}
E.~Skrettingland.
\newblock Quantum harmonic analysis on lattices and Gabor multipliers.
\newblock {\em J. Fourier Anal. Appl.}, 26:48 (2020).



\end{thebibliography}
\end{document}